\documentclass{amsart}
\newtheorem{theorem}{Theorem}[section]
\newtheorem{lemma}[theorem]{Lemma}

\theoremstyle{definition}
\newtheorem{defn}[theorem]{Definition}
\newtheorem{ex}[theorem]{Example}

\newtheorem{conj}[theorem]{Conjecture}

\theoremstyle{remark}

\numberwithin{equation}{section}

\usepackage[colorlinks]{hyperref}

\begin{document}
	\title{SOME GENERALIZED JORDAN MAPS ON TRIANGULAR RINGS FORCE ADDITIVITY}

    	\author{Sk Aziz}
    \address{Department of Mathematics, Indian Institute of Technology Patna, Patna-801106}
    \email{aziz\_2021ma22@iitp.ac.in}
    \thanks{}
    
	\author{Arindam Ghosh}
	\address{Department of Mathematics, Government Polytechnic Kishanganj, Kishanganj-855116}
	\curraddr{}
	\email{E-mail: arindam.rkmrc@gmail.com}
	\thanks{}
	
	\author{Om Prakash$^{\star}$}
	\address{Department of Mathematics, Indian Institute of Technology Patna, Patna-801106}
	\curraddr{}
	\email{om@iitp.ac.in}
	\thanks{* Corresponding author}

	\subjclass[2020]{16W10, 16S50}
	
	\keywords{Additivity, Jordan derivation, Two-sided centralizer, $(m,n)$-derivation, Triangular ring}
	
	\date{}
	
	\dedicatory{}
	
	\maketitle
\begin{abstract}
	In this paper, we show that a map $\delta$ over a triangular ring $\mathcal{T}$ satisfying
    $\delta(ab+ba)=\delta(a)b+a \tau(b)+\delta(b)a+b\tau(a)$,
    for all $a,b\in \mathcal{T}$ and for some maps $\tau$ over $\mathcal{T}$ satisfying
    $\tau(ab+ba)=\tau(a)b+a \tau(b)+\tau(b)a+b\tau(a)$,    
     is additive. Also, it is shown that a map $T$ on $\mathcal{T}$ satisfying
    $T(ab)=T(a)b=aT(b)$,
    for all $a,b\in \mathcal{T}$, is additive. Further, we establish that if a map $D$ over $\mathcal{T}$ satisfies
    $(m+n)D(ab)=2mD(a)b+2naD(b)$,
    for all $a,b\in \mathcal{T}$ and integers $m,n\geq 1$, then $D$ is additive.
\end{abstract}

\section{Introduction}
    Let $R$ be a ring. A map $f: R \rightarrow R$ is said to be additive if
    \begin{align*}
    	f(a+b)=f(a)+f(b)
    \end{align*}
    for all $a, b\in R$. In 1957, Herstein \cite{herstein1957jordan} introduced Jordan derivation over rings as an additive map $\tau:R \rightarrow R$ satisfies
    \begin{align*}
    	\tau(a^2)=\tau(a)a+a \tau(a)
    \end{align*}
    for all $a\in R$. He also proved that any Jordan derivation over a prime ring is a derivation with some torsion restrictions. In 2003, Jing and Lu \cite{jing2003generalized} introduced generalized Jordan derivation and proved that every generalized Jordan derivation is a generalized derivation over a $2$-torsion-free prime ring. Recall that an additive map $\delta: R \rightarrow R$ is said to be a generalized Jordan derivation if
    \begin{align*}
    	\delta(a^2)=\delta(a)a+a \tau(a)
    \end{align*}
    for all $a\in R$ and for some Jordan derivation $\tau: R \rightarrow R$. In 1952, Wendel \cite{wendel1952left} introduced the concept of a centralizer. An additive map $T: R \rightarrow R$ is said to be a left centralizer if
    \begin{align*}
    	T(ab)=T(a)b
    \end{align*}
    for all $a,b\in R$. Similarly, we define the right centralizer over a ring. An additive map $T: R \rightarrow R$ is said to be a two-sided centralizer if $T$ is both left and right centralizer.
    In 2014, Ali and Fo{\v{s}}ner  \cite{ali2014generalized} introduced the concept of $(m,n)$-derivation. Let $m,n\geq 0$ be integers. Then an additive map $D:R \rightarrow R$ is said to be an $(m,n)$-derivation if
    \begin{align*}
    	(m+n)D(ab)=2mD(a)b+2naD(b),
    \end{align*}
    for all $a,b\in R$. In fact, for $m=1$ and $n=1$, $D$ is a normal derivation on a $2$-torsion free ring $R$.\\

    Let $n\geq 2$ be any integer. Then a ring $R$ is said to be $n$-torsion free if $na=0$ for some $a\in R$ implies $a=0$. It is an interesting question that when a map over a ring (satisfying some functional equations) is additive. The question was first raised by Rickart \cite{rickart1948one} in 1948. He showed that any bijective and multiplicative mapping of a Boolean ring $\textsc{B}$ onto any arbitrary ring $\textsc{S}$ is additive. Later, Johnson made an extensive study of these results in \cite{johnson1958rings}. In 1969, Martindale III \cite{martindale1969multiplicative} proved that any multiplicative isomorphism of a ring $R$ onto an arbitrary ring $S$ is additive under the existence of a family of idempotent elements in $R$ satisfying certain conditions. In 1991, by assuming Martindale's conditions, Daif \cite{daif1991multiplicative} proved that any multiplicative derivation of a ring $R$ is additive.
    Later, in 2009, Wang \cite{wang2009additivity} proved that any $n$-multiplicative derivation $d$ of $R$ is additive by using Martindale's conditions where $n>1$ is an integer. He proved that any $n$-multiplicative isomorphism or $n$-multiplicative derivation of a standard operator algebra $A$ over a Banach space $X$ with $dim(X)\geq 2$, is additive.\\

    Again, in 2011, Wang \cite{wang2011additivity} proved that any $n$-multiplicative derivation $d$ of a triangular ring $\mathcal{T}$ is additive, with some assumptions on $\mathcal{T}$. In 2014, Ferreira \cite{ferreira2014multiplicative} proved that every $m$-multiplicative isomorphism from a $n$-triangular matrix ring $\mathcal{T}$ onto any ring $\mathcal{S}$ is additive with some assumptions on $\mathcal{T}$ where $n,m>1$ are integers. He also showed that every $m$-multiplicative derivation of $\mathcal{T}$ is additive. In 2015, Ferreira \cite{ferreira2015jordan} again revisit and proved that a map $\tau$ on a triangular ring $\mathcal{T}$ satisfying
    \begin{align*}
    	\tau(ab+ba)=\tau(a)b+a \tau(b)+\tau(b)a+b\tau(a)
    \end{align*}
    for all $a,b\in \mathcal{T}$, is additive with some assumptions on $\mathcal{T}$. Moreover, if $\mathcal {T}$ is $2$-torsion free, then $\tau$ is a Jordan derivation. Ferreira's result \cite{ferreira2015jordan} motivated us to work to find out the conditions under which a multiplicative generalized Jordan derivation is additive on a triangular matrix ring. This paper provides an affirmative answer to the above question.\\

    In 2014, El-Sayed et al. \cite{el2014multiplicativity} proved that every map $T$ on a prime ring with a non-trivial idempotent satisfying
    \begin{equation}
    	T(ab)=T(a)b
    \end{equation}
    for all $a,b \in R$, is additive. Note that a ring $R$ is said to be a prime ring if $aRb=0$ for some $a, b\in R$ implies either $a=0$ or $b=0$ and is said to be a semi-prime ring if $aRa=0$ for some $a\in R$ implies that $a=0$. Thus, every prime ring is semi-prime, but the converse is not true.
    Motivated by the work of  El-Sayed et al. \cite{el2014multiplicativity}, we prove that a map $T$ on $\mathcal{T}$ satisfying
    \begin{align*}
    	T(ab)=T(a)b=a T(b),
    \end{align*}
    for all $a,b\in \mathcal{T}$, is additive. \\

    Further, we prove that any map $D$ on $\mathcal{T}$ which satisfies
    \begin{align*}
    	(m+n)D(ab)=2mD(a)b+2naD(b),
    \end{align*}
    for all $a,b\in \mathcal{T}$ and for some integers $m,n\geq 1$, is additive.\\

    Throughout the work, we use the definition of triangular algebra given by Wang \cite{wang2011additivity}.

    \begin{defn}
    	\label{def1}
    	Let $A$ and $B$ be two rings, and $M$ be an $(A, B)$-bimodule such that\\
    	
    	(a) $M$ is faithful as left $A$-module and right $B$-module;
    	
    	(b) If $m \in M$ is such that $AmB = 0$ implies $m = 0$.\\
    	
    	Then the ring
    	\begin{align*}
    		\mathcal{T}=\text{Tri}(A,M,B)=  \Biggl\{ \left( \begin{array}{ccc}
    			a & m \\
    			& b \end{array}  \right)\vert \enspace a \in A,\, b \in B,\, m \in M    \Biggl\},
    	\end{align*}
    	
    	under usual matrix addition and multiplication is said to be a triangular ring.
    \end{defn}

    Let $\mathcal{T}_{11}=  \Biggl\{\left( \begin{array}{ccc}
    	a & 0 \\
    	& 0 \end{array}  \right)\vert \enspace a \in A   \Biggl\}$,\hspace{0.2cm} $\mathcal{T}_{12}=  \Biggl\{\left( \begin{array}{ccc}
    	0 & m \\
    	& 0 \end{array}  \right)\vert \enspace m \in M   \Biggl\}$ and

    $\mathcal{T}_{22}=  \Biggl\{\left( \begin{array}{ccc}
    	0 & 0 \\
    	& b \end{array}  \right)\vert \enspace b \in B   \Biggl\}$.

    Then $\mathcal{T}=\mathcal{T}_{11} \oplus \mathcal{T}_{12} \oplus \mathcal{T}_{22}$. Henceforth, $t_{ij}\in \mathcal{T}_{ij}$. Also, $t_{ij}t_{kl}=0$, if $j\neq k$ where $j,~k\in \{1,2\}$.

    \section{GENERALIZED JORDAN DERIVATIONS ON TRIANGULAR RINGS}

    \begin{theorem}
    	\label{der1}
    	Let $\mathcal{T}$ be a triangular ring with
    	conditions:\\
    	
    	(i) $aA=0$ for some $a\in A$ implies $a=0$;
    	
    	(ii) $bB=0$ for some $b\in B$ implies $b=0$;
    	
    	(iii) $A$ and $B$ are rings with identity.\\

    	If a function $\delta:\mathcal{T}  \rightarrow \mathcal{T}$ satisfies
    	\begin{align*}
    		\delta(ab+ba)=\delta(a)b+a \tau(b)+\delta(b)a+b\tau(a),
    	\end{align*}
    	
    	for all $a,b\in \mathcal{T}$ where $\tau:\mathcal{T} \rightarrow \mathcal{T}$ satisfies
    	\begin{align*}
    		\tau(ab+ba)=\tau(a)b+a \tau(b)+\tau(b)a+b\tau(a),
    	\end{align*}

    	for all $a,b\in \mathcal{T}$, then $\delta$ is additive. Moreover, if $\mathcal{T}$ is $2$-torsion free, then $\delta$ is a generalized Jordan derivation.
    \end{theorem}

    Since $\tau:\mathcal{T}  \rightarrow \mathcal{T}$ satisfies all the conditions in Theorem $3.1$ of \cite{ferreira2015jordan}, $\tau$ is additive.
    In this section, we frequently use all three conditions for the triangular ring $\mathcal{T}$, the conditions on the maps $\delta$ and $\tau$ described in Theorem \ref{der1}, and the additivity of $\tau$ without mentioning them. Note that $\delta(0)=0$. Before proving Theorem \ref{der1}, we first discuss several needful lemmas.

    \begin{lemma}
    	\label{lem1}
    	Let $a\in \mathcal{T}_{11}$, $b\in \mathcal{T}_{22}$ and $m\in \mathcal{T}_{12}$. Then
    	\begin{equation}
    		\begin{aligned}
    			&(i)~~\delta(a+m)=\delta(a)+\delta(m),\\
    			&(ii)~~\delta(b+m)=\delta(b)+\delta(m).
    		\end{aligned}
    	\end{equation}
    \end{lemma}

    \begin{proof}
    	Let $t_2 \in \mathcal{T}_{22}$. Then
    	\begin{equation}
    		\label{eq:der1}
    		\begin{aligned}
    			&\delta[(a+m)t_2+t_2(a+m)]\\
    			&=\delta(a+m)t_2+(a+m)\tau(t_2)+\delta(t_2)(a+m)+t_2 \tau(a+m).
    		\end{aligned}
    	\end{equation}
    	
    	On the other hand,
    	\begin{equation}
    		\label{eq:der2}
    		\begin{aligned}
    			&\delta[(a+m)t_2+t_2(a+m)]\\
    			&=\delta(mt_2)=\delta(0)+\delta(mt_2+0)\\
    			&=\delta(a t_2+t_2 a)+\delta(mt_2+t_2 m)\\
    			&=\delta(a)t_2+a \tau(t_2)+\delta(t_2)a+t_2 \tau(a)\\
    			&+\delta(m)t_2+m \tau(t_2)+\delta(t_2)m+t_2 \tau(m).
    		\end{aligned}
    	\end{equation}
    	
    	Since $\tau$ is additive, comparing \eqref{eq:der1} and \eqref{eq:der2}, we have
    	\begin{equation}
    		\begin{aligned}
    			&[\delta(a+m)-\delta(a)-\delta(m)]t_2=0\\
    			& \implies [\delta(a+m)-\delta(a)-\delta(m)]_{12}t_2=0\\
    			& \text{and}~~ [\delta(a+m)-\delta(a)-\delta(m)]_{22}t_2=0.
    		\end{aligned}
    	\end{equation}
    	
    	Using conditions (b) and (ii) of Definition \ref{def1} and Theorem \ref{der1}, respectively, we have
    	
    	\begin{equation}
    		\begin{aligned}
    			& [\delta(a+m)-\delta(a)-\delta(m)]_{12}=0\\
    			& \text{and}~~ [\delta(a+m)-\delta(a)-\delta(m)]_{22}=0.
    		\end{aligned}
    	\end{equation}
    	
    	Now, in order to prove $[\delta(a+m)-\delta(a)-\delta(m)]_{11}=0$, let $n\in \mathcal{T}_{12}$. Then
    	\begin{equation}
    		\label{eq:der3}
    		\begin{aligned}
    			& \delta[(a+m)n+n(a+m)]=\delta(an)\\
    			&=\delta(an+na)+\delta(mn+nm)\\
    			&=\delta(a)n +a\tau(n) +\delta(n)a +n\tau(a) +\delta(m)n +m\tau(n)+ \delta(n)m +n\tau(m).\\
    		\end{aligned}
    	\end{equation}
    	
    	Also, we have
    	\begin{equation}
    		\label{eq:3}
    		\begin{aligned}
    			& \delta[(a+m)n+n(a+m)]\\
    			& = \delta(a+m)n+ (a+m)\tau(n)+\delta(n)(a+m)+n\tau(a+m)\\
    			&=\delta(a+m)n+ (a+m)\tau(n)+\delta(n)(a+m)+n(\tau(a)+\tau(m))~(\text{By additivity of}~\tau).
    		\end{aligned}
    	\end{equation}
    	
    	Comparing the equalities \eqref{eq:der3} and \eqref{eq:3},
    	
    	\begin{equation}
    		\begin{aligned}
    			&  [\delta(a+m)-\delta(a)-\delta(m)]_{11}n=0\\
    			&\implies [\delta(a+m)-\delta(a)-\delta(m)]_{11}=0~(\text{by condition (a) of Definition \ref{def1}}).
    		\end{aligned}
    	\end{equation}
    	
    	Hence, $\delta(a+m)=\delta(a)+\delta(m)$. Similarly, we can prove that $\delta(b+m)=\delta(b) + \delta(m)$.
    \end{proof}

    \begin{lemma}
    	\label{lem2}
    	Let $a \in \mathcal{T}_{11}$, $b\in \mathcal{T}_{22}$ and $m,n \in \mathcal{T}_{12}$. Then
    	\begin{equation}
    		\delta(am+nb)=\delta(am)+\delta(nb).
    	\end{equation}
    \end{lemma}

    \begin{proof}
    	Since $am + nb = (a+n)(m+b) + (m+b)(a+n)$, by Lemma \ref{lem1} and the additivity of $\tau$, we have
    	
    	\begin{equation}
    		\begin{aligned}
    			\delta(am+nb)
    			&= \delta[(a+n)(m+b) + (m+b)(a+n)]\\
    			&=\delta(a+n)(m+b) + (a+n)\tau(m+b)\\
    			& +\delta(m+b)(a+n) + (m+b)\tau(a+n)\\
    			&=(\delta(a)+ \delta(n))(m+b) +(a+n)(\tau(m)+\tau(b))\\
    			& +(\delta(m)+\delta(b))(a+n) + (m+b)(\tau(a)+\tau(n))\\
    			&=\delta(am+ma) + \delta(nb+bn) + \delta(mn+nm) + \delta(ab+ba)\\
    			&= \delta(am) +\delta(nb).
    		\end{aligned}
    	\end{equation}
    \end{proof}

    \begin{lemma}
    	\label{lem3}
    	Let $m,n \in \mathcal{T}_{12}$. Then
    	\begin{equation}
    		\delta(m+n) = \delta(m)+\delta(n).
    	\end{equation}
    \end{lemma}

    \begin{proof}
    	Since we assume both $A$ and $B$ are rings with identity in	Theorem \ref{der1}, putting $a=1$ and $b=1$ in Lemma \ref{lem2}, we have the desired result.
    \end{proof}

    \begin{lemma}
    	\label{lem4}
    	Let $a_{1}, a_{2}, \in \mathcal{T}_{11}$ and $b_{1}, b_{2}, \in \mathcal{T}_{22}$. Then
    	\begin{equation}
    		\begin{aligned}
    			& (i)~ \delta(a_{1}+a_{2}) = \delta(a_{1})+\delta(a_{2}),\\
    			& (ii)~\delta(b_{1}+b_{2}) = \delta(b_{1})+\delta(b_{2}).
    		\end{aligned}
    	\end{equation}
    \end{lemma}

    \begin{proof}
    	Let $t_{2} \in \mathcal{T}_{22}$. Then
    	\begin{equation}
    		\label{eq:der4}
    		\begin{aligned}
    			&\delta(a_{1}+a_{2})t_{2} + (a_{1}+a_{2})\tau (t_{2}) +  \delta(t_{2})(a_{1}+a_{2})+t_{2}\tau(a_{1}+a_{2})\\
    			&=\delta[(a_{1}+a_{2})t_{2}+t_{2}(a_{1}+a_{2})]=\delta(0)=0\\
    			&=\delta(a_{1}t_{2}+t_{2}a_{1}) + \delta(a_{2}t_{2}+t_{2}a_{2})\\
    			&=\delta(a_{1})t_{2} +a_{1} \tau(t_{2}) + \delta(t_{2})a_{1} +t_{2}\tau(a_{1})\\
    			&+\delta(a_{2})t_{2} +a_{2}\tau(t_{2}) +\delta(t_{2})a_{2} + t_{2}\tau(a_{2}).
    		\end{aligned}
    	\end{equation}
    	
    	Since $\tau$ is additive, from \eqref{eq:der4}, we have
    	\begin{equation}
    		\begin{aligned}
    			&[\delta(a_{1}+a_{2}) - \delta(a_{1}) - \delta(a_{2})]t_{2}=0\\
    			&\implies [\delta(a_{1}+a_{2})-\delta(a_{1})-\delta(a_2)]_{12}t_{2}=0\\
    			& \text{and}~~ [\delta(a_{1}+a_{2})-\delta(a_{1})-\delta(a_{2})]_{22}t_{2}=0.
    		\end{aligned}
    	\end{equation}

    	Using conditions $(b)$ and $(ii)$ of Definition \ref{def1} and Theorem \ref{der1}, respectively, we have
    	\begin{equation}
    		\begin{aligned}
    			& [\delta(a_{1}+a_{2})-\delta(a_{1})-\delta(a_{2})]_{12}=0\\
    			& \text{and}~~ [\delta(a_{1}+a_{2})-\delta(a_{1})-\delta(a_{2})]_{22}=0.
    		\end{aligned}
    	\end{equation}
    	
    	Now, to prove $[\delta(a_{1}+a_{2})-\delta(a_{1})-\delta(a_{2})]_{11}=0$, let $m\in \mathcal{T}_{12}$. Then
    	\begin{equation}
    		\label{eq:der5}
    		\begin{aligned}
    			&\delta(a_{1}+a_{2})m + (a_{1}+a_{2})\tau(m) +\delta(m)(a_{1}+a_{2}) +m\tau(a_{1}+a_{2})\\
    			&= \delta[(a_{1}+a_{2})m + m(a_{1}+a_{2})]\\
    			&=\delta[(a_{1}m+ma_{1}) + (a_{2}m+ma_{2})]\\
    			&=\delta(a_{1}m+ma_{1}) + \delta(a_{2}m+ma_{2}) ~(\text{By Lemma \ref{lem3}})\\
    			&=\delta(a_{1})m + a_{1}\tau(m) + \delta(m)a_{1} + m\tau(a_{1}) + \delta(a_{2})m +  a_{2}\tau(m) + \delta(m)a_{2} +m\tau(a_{2}).\\
    		\end{aligned}
    	\end{equation}

    	By \eqref{eq:der5} and the additivity of $\tau$, we see that
    	\begin{equation}
    		\begin{aligned}
    			&  [\delta(a_{1}+a_{2})-\delta(a_1)-\delta(a_{2})]m=0\\
    			&\implies [\delta(a_{1}+a_{2})-\delta(a_1)-\delta(a_{2})]_{11} m=0\\
    			&\implies [\delta(a_{1}+a_{2})-\delta(a_1)-\delta(a_{2})]_{11}=0~(\text{by condition (a) of Definition \ref{def1}}).
    		\end{aligned}
    	\end{equation}
    	
    	Hence, $\delta(a_{1}+a_{2}) = \delta(a_{1})+\delta(a_{2})$. In a similar way, we can prove that $\delta(b_{1}+b_{2}) = \delta(b_{1})+\delta(b_{2}).$
    \end{proof}

    \begin{lemma}
    	\label{lem5}
    	Let $a\in \mathcal{T}_{11}$, $b \in \mathcal{T}_{22}$ and $m \in \mathcal{T}_{12}$. Then
    	\begin{equation}
    		\delta(a+m+b) = \delta(a)+\delta(m) +\delta(b).
    	\end{equation}
    \end{lemma}

    \begin{proof}
    	Let $t_{2} \in \mathcal{T}_{22}$. Then
    	\begin{equation}
    		\begin{aligned}
    			&\delta(a+m+b)t_2 + (a+m+b)\tau(t_2) + \delta(t_2)(a+m+b)+ t_2\tau(a+m+b)\\
    			&=\delta[(a+m+b)t_2+t_2(a+m+b)]\\
    			&=\delta(mt_2+bt_2+t_2b)\\
    			&=\delta(bt_2+t_2b) + \delta(mt_2)~(\text{By Lemma \ref{lem1}})\\
    			&=\delta(bt_2+t_2b) + \delta(t_2m+mt_2) + \delta(at_2+t_2a)\\
    			&=\delta(b)t_2+ b\tau(t_2) + \delta(t_2)b + t_2\tau(b) + \delta(t_2)m + t_2 \tau(m) \\
    			&+ \delta(m)t_2 +m\tau(t_2) + \delta(a)t_2 + a\tau(t_2 ) + \delta(t_2)a + t_2 \tau(a)\\
    			& \implies [\delta(a+m+b) - \delta(a) - \delta(m) - \delta(b)]t_2 =0 ~(\text{Using additivity of}~ \tau)\\
    			& \implies [\delta(a+m+b) - \delta(a) - \delta(m) - \delta(b)]_{12}t_2 =0\\
    			& \text{and}~[\delta(a+m+b) - \delta(a) - \delta(m) - \delta(b)]_{22}t_2 =0.
    		\end{aligned}
    	\end{equation}	
    	
    	Using conditions $(b)$ and $(ii)$ of Definition \ref{def1} and Theorem \ref{der1}, we have
    	\begin{equation}
    		\label{eq:6}
    		\begin{aligned}
    			& [\delta(a+m+b) - \delta(a) - \delta(m) - \delta(b)]_{12}=0\\
    			& \text{and}~[\delta(a+m+b) - \delta(a) - \delta(m) - \delta(b)]_{22}=0.
    		\end{aligned}
    	\end{equation}

    	Let $t_1 \in \mathcal{T}_{11}$. Then
    	\begin{equation}
    		\label{eq:7}
    		\begin{aligned}
    			&\delta(a+m+b)t_1 + (a+m+b)\tau(t_1) + \delta(t_1)(a+m+b)+ t_1\tau(a+m+b)\\
    			&=\delta[(a+m+b)t_1+t_1(a+m+b)]\\
    			&=\delta(at_1+t_1a+t_1 m)\\
    			&=\delta(at_1+t_1 a) + \delta(t_1m)~(\text{By Lemma \ref{lem1}})\\
    			&=\delta(a t_1+t_1 a) + \delta(t_1 m+mt_1) + \delta(bt_1+t_1 b)\\
    			&=\delta(a)t_1+ a\tau(t_1) + \delta(t_1)a + t_1 \tau(a) + \delta(t_1)m + t_1 \tau(m)\\
    			&+ \delta(m)t_1 + m\tau(t_1)  + \delta(b) t_1 + b\tau(t_1) +  \delta(t_1)b + t_1 \tau(b)\\
    			&\implies [\delta(a+m+b) - \delta(a) - \delta(m) - \delta(b)]t_1 =0\\
    			& \implies [\delta(a+m+b) - \delta(a) - \delta(m) - \delta(b)]_{11}t_1 =0\\
    			& \implies [\delta(a+m+b) - \delta(a) - \delta(m) - \delta(b)]_{11}=0 \\
    			& (\text{By condition (i) of Theorem \ref{der1}}).
    		\end{aligned}
    	\end{equation}
    	
    	Hence, by \eqref{eq:6} and \eqref{eq:7}, $\delta(a+m+b) = \delta(a) + \delta(m) + \delta(b)$.
    \end{proof}

    Now, we are ready to prove Theorem \ref{der1}.

    \begin{proof}[Proof of Theorem \ref{der1}]
    	Let $t \in \mathcal{T}$ and $u \in \mathcal{T}$. Then $t = t_{11} + t_{12} + t_{22}$ and $u = u_{11}+ u_{12} + u_{22}$ where $t_{ij}, u_{ij}\in \mathcal{T}_{ij}$ and $i,~j \in \{1,2\}$. Now,
    	\begin{equation}
    		\begin{aligned}
    			\delta(t+u) &= \delta((t_{11} + t_{12} + t_{22}) + (u_{11} + u_{12} + u_{22}))\\
    			&=\delta((t_{11} + u_{11}) + (t_{12} + + u_{12}) + (t_{22} + u_{22}))\\
    			&=\delta(t_{11} + u_{11}) + \delta(t_{12} + u_{12})  + \delta(t_{22} + u_{22})~(\text{By Lemma \ref{lem5}})\\
    			&=\delta(t_{11}) + \delta(u_{11}) + \delta(t_{12}) + \delta(u_{12}) + \delta(t_{22}) +\delta(u_{22})\\
    			&~(\text{By Lemma \ref{lem3} and \ref{lem4}})\\
    			&=\delta(t_{11}) + \delta(t_{12})+ \delta(t_{22}) + \delta(u_{11})  + \delta(u_{12}) +\delta(u_{22})\\
    			&=\delta(t_{11}+ t_{12} + t_{22}) + \delta(u_{11}+ u_{12} + u_{22})~(\text{By Lemma \ref{lem5}}) \\
    			&=\delta(t) + \delta(u).\\
    		\end{aligned}
    	\end{equation}

    	Therefore, $\delta$ is additive.\\
    	
    	Let $\mathcal{T}$ be $2$-torsion free. By Theorem $3.1$ of \cite{ferreira2015jordan}, $\tau$ is a Jordan derivation. For any $t \in \mathcal{T}$,
    	\begin{equation}
    		\begin{aligned}
    			&2\delta(t^2)= \delta(2t^2) = \delta(tt +tt)
    			=2[\delta(t)t + t\tau(t)]\\
    			&\implies \delta(t^2)=\delta(t)t + t\tau(t).
    		\end{aligned}
    	\end{equation}
    	
    	Thus, $\delta$ is a generalized Jordan derivation.
    \end{proof}

    The following example shows that dropping condition $(iii)$ of Theorem \ref{def1}
    does not make any difference in the conclusion of this theorem.

    \begin{ex}
    	\label{ex1}
    	Let $\mathcal{T}=\text{Tri}(2\mathbb{Z},\mathbb{Z},3\mathbb{Z})=  \Biggl\{\left( \begin{array}{ccc}
    		a & m \\
    		& b \end{array}  \right)\vert \enspace a \in 2\mathbb{Z},\, b \in 3\mathbb{Z},\, m \in \mathbb{Z}    \Biggl\}$ where $2\mathbb{Z}$ is the ring of even integers, $3\mathbb{Z}$ is the ring of integers multiplied by $3$ and $\mathbb{Z}$ is the $(2\mathbb{Z},3\mathbb{Z})$-bimodule of integers. Note that $2\mathbb{Z}$ and $3\mathbb{Z}$ are rings without identity. Define $\tau : \mathcal{T}  \rightarrow \mathcal{T}$ as
    	\begin{align*}
    		\tau  \left( \begin{array}{ccc}
    			a & m \\
    			& b \end{array}  \right)= \left( \begin{array}{ccc}
    			0 & \alpha a+\beta m-\alpha b \\
    			& 0 \end{array}  \right),
    	\end{align*}
    	where $\alpha, \beta \in \mathbb{Z}$. Then it is easy to check that $\tau$ satisfies all the assumptions of Theorem $3.1$ of \cite{ferreira2015jordan}. Hence, $\tau$ is additive and also a Jordan derivation. Define $\delta:\mathcal{T} \rightarrow \mathcal{T}$ as
    	\begin{align*}
    		\delta  \left( \begin{array}{ccc}
    			a & m \\
    			& b \end{array} \right )= \left( \begin{array}{ccc}
    			a \gamma & \alpha a+(\beta+\gamma) m+\mu b \\
    			& wb \end{array}  \right),
    	\end{align*}
    	where $\gamma, \mu$ and  $w \in \mathbb{Z}$. Then it is a routine exercise to check that $\delta$ satisfies all the conditions except (iii) of Theorem \ref{der1} with the associated map $\tau$ just mentioned. Also, we can quickly check that $\delta$ is additive and hence a generalized Jordan derivation.
    \end{ex}

    Example \ref{ex1} motivates us to post Theorem \ref{der1} without assuming condition $(iii)$ as a conjecture.

    \begin{conj}
    	Let $\mathcal{T}$ be a triangular ring with
    	conditions:\\
    	
    	(i) $aA=0$ for some $a\in A$ implies $a=0$;
    	
    	(ii) $bB=0$ for some $b\in B$ implies $b=0$. \\

    	If a map $\delta:\mathcal{T}  \rightarrow \mathcal{T}$ satisfies
    	\begin{align*}
    		\delta(ab+ba)=\delta(a)b+a \tau(b)+\delta(b)a+b\tau(a),
    	\end{align*}
    	
    	for all $a,b\in \mathcal{T}$ where $\tau:\mathcal{T} \rightarrow \mathcal{T}$ satisfies
    	\begin{align*}
    		\tau(ab+ba)=\tau(a)b+a \tau(b)+\tau(b)a+b\tau(a),
    	\end{align*}

    	for all $a,b\in \mathcal{T}$, then $\delta$ is additive. Moreover, if $\mathcal{T}$ is $2$-torsion free, then $\delta$ is a generalized Jordan derivation.
    \end{conj}

    \section{TWO-SIDED CENTRALIZERS ON TRIANGULAR RINGS}
    \begin{theorem}
    	\label{thm3.1}
    	Let $\mathcal{T}$ be a triangular ring with
    	conditions:\\
    	
    	(i) $Aa=0$ for some $a\in A$ implies $a=0$,
    	
    	(ii) $bB=0$ for some $b\in B$ implies $b=0$,
    	
    	(iii) $aA=0$ for some $a\in A$ implies $a=0$.\\
    	
    	If a map $T:\mathcal{T}  \rightarrow \mathcal{T}$ satisfies
    	\begin{align*}
    		T(ab)=T(a)b=aT(b),
    	\end{align*}
    	
    	for all $a,b\in \mathcal{T}$, then $T$ is additive. Moreover, $T$ is a two-sided centralizer.
    \end{theorem}


    \begin{lemma}
    	\label{lem3.2}
    	Let $a\in \mathcal{T}_{11}$, $b\in \mathcal{T}_{22}$ and $m\in \mathcal{T}_{12}$. Then
    	\begin{equation}
    		\label{eq:3.1}
    		\begin{aligned}
    			&(i)~T(a+m)=T(a)+T(m),\\
    			&(ii)~T(b+m)=T(b)+T(m).
    		\end{aligned}
    	\end{equation}
    \end{lemma}

    \begin{proof}
    	Let $t_2 \in \mathcal{T}_{22}$. Now, we have
    	\begin{equation}
    		\label{eq:3.2}
    		\begin{aligned}
    			&T[(a+m)t_2]=T(a+m)t_2.
    		\end{aligned}
    	\end{equation}
    	
    	Also,
    	\begin{equation}
    		\label{eq:3.3}
    		\begin{aligned}
    			T[(a+m)t_2]=T(mt_2)&=0+T(mt_2)\\
    			&=T(at_2)+T(mt_2)\\
    			&=T(a)t_2+T(m)t_2.
    		\end{aligned}
    	\end{equation}
    	
    	Comparing identities \eqref{eq:3.2} and \eqref{eq:3.3},
    	\begin{equation}
    		\label{eq:3.4}
    		\begin{aligned}
    			&[T(a+m)-T(a)-T(m)]t_2=0\\
    			&\implies [T(a+m)-T(a)-T(m)]_{12}t_2=0\\
    			&\text{and}~[T(a+m)-T(a)-T(m)]_{22}t_2=0\\
    			&\implies [T(a+m)-T(a)-T(m)]_{12}=0 ~(\text{By condition (b) of Definition \ref{def1}})\\
    			&\text{and}~[T(a+m)-T(a)-T(m)]_{22}=0~(\text{By condition (ii) of Theorem \ref{thm3.1}}).
    		\end{aligned}
    	\end{equation}
    	
    	Let $n\in \mathcal{T}_{12}$. We have
    	\begin{equation}
    		\label{eq:3.5}
    		\begin{aligned}
    			&T[(a+m)n]=T(a+m)n.
    		\end{aligned}
    	\end{equation}
    	
    	Also,
    	\begin{equation}
    		\label{eq:3.6}
    		\begin{aligned}
    			&T[(a+m)n] =T(an)+0=T(an)+T(mn)=T(a)n+T(m)n.
    		\end{aligned}
    	\end{equation}
    	
    	Comparing \eqref{eq:3.5} and \eqref{eq:3.6}, we have
    	\begin{equation}
    		\label{eq:3.7}
    		\begin{aligned}
    			&[T(a+m)-T(a)-T(m)]n=0\\
    			&\implies [T(a+m)-T(a)-T(m)]_{11}n=0\\
    			&\implies [T(a+m)-T(a)-T(m)]_{11}=0~(\text{By Condition (a) of Definition \ref{def1}}).
    		\end{aligned}
    	\end{equation}
    	Hence, by using \eqref{eq:3.4} and \eqref{eq:3.7}, we get $(i)$ of Lemma \ref{lem3.2}. Similarly, we have $(ii)$ of Lemma \ref{lem3.2}.
    \end{proof}

    \begin{lemma}
    	\label{lem3.3}
    	Let $a \in \mathcal{T}_{11}, b \in \mathcal{T}_{22}$ and $m, n \in \mathcal{T}_{12}$. Then
    	\begin{equation}
    		T(a m+n b)=T(a m)+T(n b).
    	\end{equation}
    \end{lemma}

    \begin{proof}
    	Since $a m + nb= (a + n)(b + m)$, we have
    	\begin{equation}
    		\label{eq:3.9}
    		\begin{aligned}
    			T(a m + nb)&= T((a+ n)(b + m))\\
    			&=(a + n)T(b + m)\\
    			&=(a + n)(T(b) + T(m))~(\text{By (ii) of Lemma \ref{lem3.2}}).
    		\end{aligned}
    	\end{equation}
    	
    	Also,
    	\begin{equation}
    		\label{eq:3.10}
    		\begin{aligned}
    			T(a m)+T(n b)&=T(a m)+T(n m)+T(n b)+T(a b)\\
    			&=a T(m)+n T(m)+n T(b)+a T(b)\\
    			&=(a+n) T(m)+(a+n) T(b)\\
    			&=(a+n)(T(m)+T(b)).
    		\end{aligned}
    	\end{equation}
    	
    	Comparing \eqref{eq:3.9} and \eqref{eq:3.10}, we have the desired result.
    \end{proof}

    \begin{lemma}
    	\label{lem3.4}
    	Let $b \in \mathcal{T}_{22}$ and $m,n \in \mathcal{T}_{12}$. Then
    	\begin{equation}
    		T(m b+n b)=T(m b) + T(n b).
    	\end{equation}
    \end{lemma}

    \begin{proof}
    	Let  $a \in \mathcal{T}_{11}$. Then
    	
    	\begin{equation}
    		\label{eq:3.12}
    		\begin{aligned}
    			T[a((m+n) b) ] =a T((m+n) b ).
    		\end{aligned}
    	\end{equation}
    	
    	Also,
    	\begin{equation}
    		\label{eq:3.13}
    		\begin{aligned}
    			T[a((m+n) b) ]&=T (a (m b )+ (a n) b )\\
    			& =T(a (m b) ) + T ((a n) b )~(\text{By Lemma \ref{lem3.3}})\\
    			&=a T(m b )+a T (n b ).\\
    		\end{aligned}
    	\end{equation}
    	
    	Comparing \eqref{eq:3.12} and \eqref{eq:3.13}, we have
    	\begin{equation}
    		\label{eq:3.14}
    		\begin{aligned}
    			& a [T(mb+nb)-T (m b )-T (m b ) ]=0\\
    			& \implies a[T (mb+nb )-T (m b )-T(n b )]_{11}=0 \\
    			& \&~a[T (mb+nb )-T (m b )-T (n b ) ]_{12}=0\\
    			& \implies [T (mb+nb )-T (m b )-T(n b )]_{11}=0~(\text{By condition (i) of Theorem \ref{thm3.1}}) \\
    			& \&~[T (mb+nb)-T (m b )-T (n b ) ]_{12}=0~(\text{By condition (b) of Definition \ref{def1}}).
    		\end{aligned}
    	\end{equation}
    	
    	Let $p \in \mathcal{T}_{12}$. Then
    	\begin{equation}
    		\label{eq:3.15}
    		\begin{aligned}
    			T [p ((m+n) b ) ]=p  T((m+n) b ).
    		\end{aligned}
    	\end{equation}
    	
    	Also,
    	\begin{equation}
    		\label{eq:3.16}
    		\begin{aligned}
    			T [p ((m+n) b ) ]&=T [p (m b+n b ) ]\\
    			&=T(0)=0\\
    			&=T (p (m b ) )+T (p (n b ) )\\
    			&=pT (m b )+pT(n b ).
    		\end{aligned}
    	\end{equation}
    	
    	Comparing \eqref{eq:3.15} and \eqref{eq:3.16},
    	\begin{equation}
    		\label{eq:3.17}
    		\begin{aligned}
    			& p[T (mb+nb )-T(m b )-T (n b )]=0\\
    			& \implies p[T (mb+nb )-T(m b )-T (n b )]_{22}=0\\
    			& \implies [T (mb+nb )-T (m b )-T (n b )]_{22}=0~(\text{By condition (a) of Definition \ref{def1}}).
    		\end{aligned}
    	\end{equation}
    	
    Thus the result follows from \eqref{eq:3.14} and \eqref{eq:3.17}.
    \end{proof}

    \begin{lemma}
    	\label{lem3.5}
    	Let $m, n \in \mathcal{T}_{12}$. Then
    	\begin{equation}
    		T(m+n)=T(m)+T(n).
    	\end{equation}
    \end{lemma}

    \begin{proof}
    	Let $b\in \mathcal{T}_{22}$. Then
    	\begin{equation}
    		\label{eq:3.19}
    		\begin{aligned}
    			T((m+n) b )=T(m+n) b.
    		\end{aligned}
    	\end{equation}
    	
    	Now,
    	\begin{equation}
    		\label{eq:3.20}
    		\begin{aligned}
    			T((m+n) b )&= T(m b+n b )\\
    			&= T(m b )+T (n b )~(\text{By Lemma \ref{lem3.4}})\\
    			&=T(m) b+T(n) b.
    		\end{aligned}
    	\end{equation}
    	
    	By \eqref{eq:3.19} and \eqref{eq:3.20},
    	\begin{equation}
    		\label{eq:3.21}
    		\begin{aligned}
    			&[T(m+n)-T(m)-T(n)]b=0\\
    			& \implies [T(m+n)-T(m)-T(n)]_{12} b=0\\
    			& \&~ [T(m+n)-T(m)-T(n)]_{22} b=0\\
    			& \implies [T(m+n)-T(m)-T(n)]_{12}=0~(\text{By condition (b) of Definition \ref{def1}})\\
    			& \&~ [T(m+n)-T(m)-T(n)]_{22}=0~(\text{By condition (ii) of Theorem \ref{thm3.1}}).
    		\end{aligned}
    	\end{equation}
    	
    	Let $p \in \mathcal{T}_{12}$. Then
    	\begin{equation}
    		\label{eq:3.22}
    		\begin{aligned}
    			T ((m+n) p)=T(m+n) p.
    		\end{aligned}
    	\end{equation}	
    	
    	Also,
    	\begin{equation}
    		\label{eq:3.23}
    		\begin{aligned}
    			T ((m+n) p)=T(0)=0=T(mp) + T(np)=T(m)p + T(n)p.
    		\end{aligned}
    	\end{equation}
    	
    	Comparing \eqref{eq:3.22} and \eqref{eq:3.23},
    	\begin{equation}
    		\label{eq:3.24}
    		\begin{aligned}
    			&[T(m+n)-T(m)-T(n)]p=0\\
    			& \implies [T(m+n)-T(m)-T(n)]_{11}p=0\\
    			& \implies [T(m+n)-T(m)-T(n)]_{11}=0~(\text{By condition (a) of Definition \ref{def1}}).
    		\end{aligned}
    	\end{equation}
    	
    	By \eqref{eq:3.21} and \eqref{eq:3.24}, we get the result.
    \end{proof}

    \begin{lemma}
    	\label{lem3.6}
    	Let $a_{1}, a_2 \in \mathcal{T}_{11}$ and $b_{1}, b_{2} \in \mathcal{T}_{22}$. Then
    	\begin{equation}
    		\begin{aligned}
    			& (i)~ T(a_{1} + a_2 )=T (a_{1} )+T (a_2 )\\
    			& (ii)~ T (b_{1}+b_{2} )=T (b_{1} )+T (b_{2} ).
    		\end{aligned}
    	\end{equation}
    \end{lemma}

    \begin{proof}
    	Let $t_{2} \in \mathcal{T}_{22}$. Then
    	\begin{equation}
    		\label{eq:3.26}
    		\begin{aligned}
    			T ( (a_{1}+a_2 ) t_{2} ) = T (a_{1}+a_2 ) t_{2}.
    		\end{aligned}
    	\end{equation}
    	
    	Now,
    	\begin{equation}
    		\label{eq:3.27}
    		\begin{aligned}
    			T ( (a_{1}+a_2 ) t_{2} ) = T (0)=T (a_{1} t_{2} )+T (a_2 t_{2} )=T (a_{1} ) t_{2}+T (a_2 ) t_{2}.
    		\end{aligned}
    	\end{equation}
    	
    	Comparing \eqref{eq:3.26} and \eqref{eq:3.27},
    	\begin{equation}
    		\label{eq:3.28}
    		\begin{aligned}
    			& [T(a_1+a_2)-T(a_1)-T(a_2)]t_2=0\\
    			& \implies [T(a_1+a_2)-T(a_1)-T(a_2)]_{12} t_2=0\\
    			& \&~[T(a_1+a_2)-T(a_1)-T(a_2)]_{22} t_2=0\\
    			& \implies  [T(a_1+a_2)-T(a_1)-T(a_2)]_{12}=0~(\text{By condition (b) of Definition \ref{def1}})\\
    			& \&~[T(a_1+a_2)-T(a_1)-T(a_2)]_{22} =0~(\text{By condition (ii) of Theorem \ref{thm3.1}}).
    		\end{aligned}
    	\end{equation}
    	
    	Let $m\in \mathcal{T}_{12}$. Then
    	\begin{equation}
    		\label{eq:3.29}
    		T ( (a_{1}+a_2 ) m )=T (a_{1}+a_2 ) m.
    	\end{equation}
    	
    	Also,
    	\begin{equation}
    		\label{eq:3.30}
    		\begin{aligned}
    			T ( (a_{1}+a_2 ) m )&=T(a_{1} m+a_2 m )\\
    			&=T (a_{1} m )+T(a_2 m )~(\text{By Lemma \ref{lem3.5}})\\
    			&=T (a_{1} ) m+T (a_2 ) m.
    		\end{aligned}
    	\end{equation}
    	
    	Comparing \eqref{eq:3.29} and \eqref{eq:3.30},
    	\begin{equation}
    		\label{eq:3.31}
    		\begin{aligned}
    			&T (a_{1}+a_2 )-T (a_{1} )-T(a_2 ) ] m=0\\
    			& \implies [T (a_{1}+a_2 )-T (a_{1} )-T(a_2 ) ]_{11} m=0 \\
    			& \implies [T (a_{1}+a_2 )-T (a_{1} )-T(a_2 ) ]_{11}=0~(\text{By condition (a) of Definition \ref{def1}}).
    		\end{aligned}
    	\end{equation}
    	
    	By \eqref{eq:3.28} and \eqref{eq:3.31}, we have $(i)$ of Lemma \ref{lem3.6}. Similarly, we can prove $(ii)$ of Lemma \ref{lem3.6}.
    \end{proof}

    \begin{lemma}
    	\label{lem3.7}
    	Let $a \in \mathcal{T}_{11}, b \in \mathcal{T}_{22}$ and $m \in \mathcal{T}_{12}$. Then
    	\begin{equation}
    		T (a+m+b )=T (a )+T(m)+T (b ).
    	\end{equation}
    \end{lemma}

    \begin{proof}
    	Let $a_{1} \in \mathcal{T}_{11}$. Then
    	\begin{equation}
    		\label{eq:3.33}
    		\begin{aligned}
    			T ( (a+m+b ) a_{1} )=T (a+m+b ) a_{1}.
    		\end{aligned}
    	\end{equation}
    	
    	Also,
    	\begin{equation}
    		\label{eq:3.34}
    		\begin{aligned}
    			T ( (a+m+b ) a_{1} )&=T (a a_{1}+m a_{1}+b a_{1} )\\
    			&=T (a a_{1} )\\
    			&=T (a a_{1} )+T (m a_{1} )+T (b a_{1} )\\
    			&=T (a ) a_{1}+T(m) a_{1}+T(b ) a_{1}.
    		\end{aligned}
    	\end{equation}
    	
    	By \eqref{eq:3.33} and \eqref{eq:3.34},
    	\begin{equation}
    		\label{eq:3.35}
    		\begin{aligned}
    			&[T (a+m+b )-T (a )-T(m)-T (b ) ] a_{1}=0\\
    			& \implies [T (a+m+b)-T (a)-T(m)-T (b)]_{11} a_{1}=0\\
    			& \implies [T (a+m+b)-T (a)-T(m)-T (b)]_{11} =0\\
    			&~(\text{By condition (iii) of Theorem \ref{thm3.1}}).
    		\end{aligned}
    	\end{equation}

    	Let $t_2 \in \mathcal{T}_{22}$. Then
    	\begin{equation}
    		\label{eq:3.36}
    		\begin{aligned}
    			T ( (a+m+b ) t_2 )=T (a+m+b ) t_2
    		\end{aligned}
    	\end{equation}
    	
    	Also,	
    	\begin{equation}
    		\label{eq:3.37}
    		\begin{aligned}
    			T ( (a+m+b ) t_2 )&=T (a t_2+m t_2+b t_2 )\\
    			&=T (m t_2+b t_2 ) \\
    			&=T(m t_2 )+T (b t_2 ) \quad [\text {By (ii) of Lemma \ref{lem3.2} }] \\
    			&=T(a t_2 )+T (m t_2 )+T (b t_2 ) \\
    			&=T(a ) t_2+T(m) t_2+T (b ) t_2.
    		\end{aligned}
    	\end{equation}

    	Comparing \eqref{eq:3.36} and \eqref{eq:3.37},
    	\begin{equation}
    		\label{eq:3.38}
    		\begin{aligned}
    			& [T (a+m+b )-T (a )-T(m)-T(b ) ] t_2=0\\
    			& \implies [T (a+m+b)-T (a)-T(m)-T (b)]_{12}t_2=0\\
    			& \&~[T (a+m+b )-T (a )-T(m)-T (b ) ]_{22}t_2=0\\
    			& \implies [T (a+m+b)-T (a)-T(m)-T (b)]_{12}=0~(\text{By condition (b) of Definition \ref{def1} })\\
    			& \&~[T (a+m+b )-T (a )-T(m)-T (b ) ]_{22}=0~(\text{By condition (ii) of Theorem \ref{thm3.1}}).
    		\end{aligned}
    	\end{equation}
    	
    	Thus, by \eqref{eq:3.35} and \eqref{eq:3.38}, we have
    	$T (a+m+b )=T (a )+T(m)+T(b )$.
    \end{proof}

    Now, we are ready to prove Theorem \ref{thm3.1}.

    \begin{proof}[Proof of Theorem \ref{thm3.1}]
    	Let $t \in \mathcal{T}$ and $u \in \mathcal{T}$. Then $t = t_{11} + t_{12} + t_{22}$ and $u = u_{11}+ u_{12} + u_{22}$ where $t_{ij}, u_{ij}\in \mathcal{T}_{ij}$ and $i,~j \in \{1,2\}$. We have
    	\begin{equation}
    		\begin{aligned}
    			T(t+u) &= T((t_{11} + t_{12} + t_{22}) + (u_{11} + u_{12} + u_{22}))\\
    			&=T((t_{11} + u_{11}) + (t_{12} + u_{12}) + (t_{22} + u_{22}))\\
    			&=T(t_{11} + u_{11}) + T(t_{12} + u_{12})  + T(t_{22} + u_{22})~(\text{By Lemma \ref{lem3.7}})\\
    			&=T(t_{11}) + T(u_{11}) + T(t_{12}) + T(u_{12}) + T(t_{22}) +T(u_{22})\\
    			&~(\text{By Lemma \ref{lem3.5} and \ref{lem3.6}})\\
    			&=T(t_{11}) + T(t_{12})+ T(t_{22}) + T(u_{11})  + T(u_{12}) +T(u_{22})\\
    			&=T(t_{11}+ t_{12} + t_{22}) + T(u_{11}+ u_{12} + u_{22})~(\text{By Lemma \ref{lem3.7}}) \\
    			&=T(t) + T(u).\\
    		\end{aligned}
    	\end{equation}

    	Hence, $T$ is additive and $T$ is a two-sided centralizer.
    \end{proof}

    Now, we provide an example which shows that every triangular ring $\mathcal{T}$ is not always a prime ring.
    \begin{ex}
    	Let $\mathcal{T}= \left( \begin{array}{ccc}
    		2\mathbb{Z} &  \mathbb{Z} \\
    		& 3\mathbb{Z} \end{array}  \right)$. Then
    	\begin{align*}
    		\left( \begin{array}{ccc}
    			0 &  1 \\
    			& 0 \end{array}  \right)
    		\left( \begin{array}{ccc}
    			2\mathbb{Z} &  \mathbb{Z} \\
    			& 3\mathbb{Z} \end{array}  \right)
    		\left( \begin{array}{ccc}
    			0 &  1 \\
    			& 0 \end{array}  \right)= \left( \begin{array}{ccc}
    			0 &  0 \\
    			& 0 \end{array}  \right),
    	\end{align*}
    	but $\left( \begin{array}{ccc}
    		0 &  1 \\
    		& 0 \end{array}  \right) \neq \left( \begin{array}{ccc}
    		0 &  0 \\
    		& 0 \end{array}  \right)$. Therefore, $\mathcal{T}$ is not a semi-prime ring and hence not a prime ring.
    \end{ex}

    \section{$(m,n)$-DERIVATIONS ON TRIANGULAR RINGS}
    \begin{theorem}
    	\label{der3}
    	Let $m>0,n>0$ be integers and $\mathcal{T}$ be a triangular ring with
    	conditions:
    	(i) $aA=0$ for some $a\in A$ implies $a=0$;
    	
    	(ii) $bB=0$ for some $b\in B$ implies $b=0$;
    	
    	(iii) $Aa=0$ for some $a\in A$ implies $a=0$;
    	
    	(iv) $\mathcal{T}$ is $mn(m+n)$-torsion free.\\
    	
    	If a mapping $D:\mathcal{T}  \rightarrow \mathcal{T}$ satisfies
    	\begin{align*}
    		(m+n) D(ab)=2 m D(a) b+2 n a D(b)
    	\end{align*}
    	for all $a,b\in \mathcal{T}$, then $D$ is additive. Moreover, $D$ is a $(m,n)$-derivation.
    \end{theorem}

    In this section, we frequently use all four conditions for the triangular ring $\mathcal{T}$, the condition on the map $D$ described in Theorem \ref{der3} without mentioning them. Note that $D(0)=0$. Before proving Theorem \ref{der3}, we have some lemmas.

    \begin{lemma}
    	\label{lem4.2}
    	Let $a \in \mathcal{T}_{11}, p \in \mathcal{T}_{12}$ and $b \in \mathcal{T}_{22}$. Then
    	\begin{equation}
    		\begin{aligned}
    			& (i)~ D (a+p )=D (a )+D(p)\\
    			&	(ii)~ D (b+p )=D (b )+D(p).
    		\end{aligned}
    	\end{equation}
    \end{lemma}

    \begin{proof}
    	Let $c\in \mathcal{T}_{22}$. Then
    	\begin{equation}
    		\label{eq:4.2}
    		\begin{aligned}
    			(m+n) D ( (a+p ) c )=2 m D (a+p ) c+2 n (a+p ) D (c).
    		\end{aligned}
    	\end{equation}
    	
    	Also,
    	\begin{equation}
    		\label{eq:4.3}
    		\begin{aligned}
    			(m+n) D ( (a+p ) c )&=(m+n) D (p c )\\
    			&=(m+n) [D (a c )+D (p c ) ]\\
    			&=2 m D (a ) c+2 n a D (c )+2 m D(p) c+2 n p D (c ).
    		\end{aligned}
    	\end{equation}
    	
    	Comparing \eqref{eq:4.2} and \eqref{eq:4.3},
    	\begin{equation}
    		\label{eq:4.4}
    		\begin{aligned}
    			& 2 m [D (a+p )-D (a )-D(p) ] c=0\\
    			& \implies [D (a+p )-D (a )-D(p) ] c=0~(\text{By condition (iv) of Theorem \ref{der3}}) \\
    			& \implies [D (a+p )-D (a )-D(p) ]_{12} c=0\\
    			& \&~[D (a+p )-D (a )-D(p) ]_{22} c=0\\
    			& \implies [D (a+p )-D (a )-D(p) ]_{12}=0~(\text{By condition (b) of Definition \ref{def1}}) \\
    			& \&~[D (a+p )-D (a )-D(p) ]_{22}=0~(\text{By condition (ii) of Theorem \ref{der3}}).
    		\end{aligned}
    	\end{equation}

    	Let $a_{1} \in \mathcal{T}_{11}$. Then
    	\begin{equation}
    		\label{eq:4.5}
    		\begin{aligned}
    			(m+n) D ( (a+p ) a_{1} )
    			=2 m D (a+p ) a_{1}+2 n (a+p ) D (a_{1} ).
    		\end{aligned}
    	\end{equation}
    	
    	Also,
    	\begin{equation}
    		\label{eq:4.6}
    		\begin{aligned}
    			(m+n) D ( (a+p ) a_{1} )&=(m+n) D (a a_{1} )\\
    			&=(m+n) [D (a a_{1} )+D (p a_{1} ) ]\\
    			&=2 m D (a ) a_{1}+2 n a D (a_{1} )+2 m D(p) a_{1}+2 n p D (a_{1} ).
    		\end{aligned}
    	\end{equation}
    	
    	Comparing \eqref{eq:4.5} and \eqref{eq:4.6}, we have
    	\begin{equation}
    		\label{eq:4.7}
    		\begin{aligned}
    			& 2 m [D (a+p )-D (a )-D(p) ] a_{1}=0\\
    			& \implies [D (a+p )-D (a )-D(p) ] a_{1}=0~(\text{By condition (iv) of Theorem \ref{der3}}) \\
    			& \implies [D (a+p )-D (a )-D(p) ]_{11} a_{1}=0 \\
    			& \implies [D (a+p )-D (a )-D(p) ]_{11}=0~(\text{By condition (i) of Theorem \ref{der3}}).
    		\end{aligned}
    	\end{equation}
    	
    	By \eqref{eq:4.4} and \eqref{eq:4.7}, we have $(i)$ of Lemma \ref{lem4.2}. Similarly, we can prove $(ii)$ of Lemma \ref{lem4.2}.	
    \end{proof}

    \begin{lemma}
    	\label{lem4.3}
    	Let $a\in \mathcal{T}_{11}$, $b\in \mathcal{T}_{22}$ and $p,q\in \mathcal{T}_{12}$. Then
    	\begin{equation}
    		D(ap+qb)=D(ap)+D(qb).
    	\end{equation}
    \end{lemma}

    \begin{proof}
    	Since $(ap+qb)=(a+q)(p+b)+(p+b)(a+q)$,
    	
    	\begin{equation}
    		\label{eq:4.9}
    		\begin{aligned}
    			(m+n)D(ap+qb)&=D((a+q)(p+b)+(p+b)(a+q))\\
    			&=2mD(a+q)(p+b)+2n(a+q)D(p+b)\\
    			&=2m(D(a)+D(q))(p+b)+2n(a+q)(D(p)+D(b))\\
    			&~(\text{By Lemma \ref{lem4.2}})\\
    			&=2mD(a)p+2naD(p)+2mD(q)b+2nqD(b)\\
    			&+2mD(q)p+2nqD(p)+2mD(a)b+2naD(b)\\
    			&=(m+n)[D(ap)+D(qb)+D(qp)+D(ab)]\\
    			&=(m+n)[D(ap)+D(qb)].
    		\end{aligned}
    	\end{equation}
    	
    	Using condition $(iv)$ of Theorem \ref{der3} and the above identity \eqref{eq:4.9}, we get the desired result.
    \end{proof}

    \begin{lemma}
    	\label{lem4.4}
    	Let $b \in \mathcal{T}_{22}$ and $p,q \in \mathcal{T}_{12}$. Then
    	\begin{equation}
    		D (p b+q b )=D (p b )+D (q b ) .
    	\end{equation}
    \end{lemma}

    \begin{proof}
    	Let $a \in \mathcal{T}_{11}$. Then
    	\begin{equation}
    		\label{eq:4.11}
    		\begin{aligned}
    			(m+n) D [a ((p+q) b ) ]=2 m D (a )(p+q) b+2 n a D ((p+q) b ).
    		\end{aligned}
    	\end{equation}
    	
    	Also,
    	\begin{equation}
    		\label{eq:4.12}
    		\begin{aligned}
    			(m+n) D [a ((p+q) b ) ]&=(m+n) D [a (p b )+ (a q ) b ]\\
    			&=(m+n)[D (ap b )+D(aq b)]~(\text{By Lemma \ref{lem4.3}})\\
    			&=2 m D (a ) (p b )+2 n a D (p b )+2 m D (a ) (q b )+2 n a D (q b ).
    		\end{aligned}
    	\end{equation}
    	
    	Comparing \eqref{eq:4.11} and \eqref{eq:4.12}, we have
    	\begin{equation}
    		\label{eq:4.13}
    		\begin{aligned}
    			& 2 n a [D ((p+q) b )-D (p b )-D (q b ) ]=0\\
    			& \implies a [D ((p+q) b )-D (p b )-D (q b ) ]=0~(\text{By condition (iv) of Theorem \ref{der3}})\\
    			& \implies a [D (pb+qb )-D (p b )-D (q b ) ]_{11}=0 \\
    			& \&~a [D ((pb+qb )-D (p b )-D (q b ) ]_{12}=0\\
    			& \implies  [D (pb+qb )-D (p b )-D (q b ) ]_{11}=0~(\text{By condition (iii) of Theorem \ref{der3}}) \\
    			& \&~ [D ((pb+qb )-D (p b )-D (q b ) ]_{12}=0~(\text{By condition (b) of Definition \ref{def1}}).
    		\end{aligned}
    	\end{equation}

    	Let $s \in \mathcal{T}_{12}$. Then
    	\begin{equation}
    		\label{eq:4.14}
    		\begin{aligned}
    			(m+n) D [s (pb+qb ) ]= 2 m D(s) (pb+qb ) +2 n s D (pb+q b ).
    		\end{aligned}
    	\end{equation}
    	
    	Also,
    	\begin{equation}
    		\label{eq:4.15}
    		\begin{aligned}
    			(m+n) D [s (pb+qb ) ]&=(m+n) D(0)\\
    			&=(m+n)D(spb)+(m+n)D(sqb)\\
    			&=2 m D(s) p b+2 n s D (p b )+2 m D(s) q b+2ns D (q b ).
    		\end{aligned}
    	\end{equation}
    	
    	Comparing \eqref{eq:4.14} and \eqref{eq:4.15},
    	\begin{equation}
    		\label{eq:4.16}
    		\begin{aligned}
    			& 2n s [D(pb+qb)-D(pb)-D(qb)]=0\\
    			& \implies s [D(pb+qb)-D(pb)-D(qb)]=0~(\text{By condition (iv) of Theorem \ref{der3}})\\
    			& \implies s [D(pb+qb)-D(pb)-D(qb)]_{22}=0\\
    			& \implies  [D(pb+qb)-D(pb)-D(qb)]_{22}=0~(\text{By condition (a) of Definition \ref{def1}}).
    		\end{aligned}
    	\end{equation}
    	
    	By \eqref{eq:4.13} and \eqref{eq:4.16}, we have the desired result.
    \end{proof}

    \begin{lemma}
    	\label{lem4.5}
    	Let $p,q \in \mathcal{T}_{12}$. Then
    	\begin{equation}
    		D(p+q)=D(p)+D(q).
    	\end{equation}
    \end{lemma}

    \begin{proof}
    	Let $b \in \mathcal{T}_{22}$. Then
    	\begin{equation}
    		\label{eq:4.18}
    		\begin{aligned}
    			(m+n) D ((p+q) b )=2mD(p+q)b+2n(p+q)D(b).
    		\end{aligned}
    	\end{equation}
    	
    	Also,
    	\begin{equation}
    		\label{eq:4.19}
    		\begin{aligned}
    			(m+n) D ((p+q) b )&=(m+n) (D (p b )+D (q b ) )~(\text{By Lemma \ref{lem4.4}})\\
    			&=2 m D(p) b+2 n p D (b )+2 m D(q) b+2 n q D (b ).
    		\end{aligned}
    	\end{equation}
    	
    	Comparing \eqref{eq:4.18} and \eqref{eq:4.19},
    	\begin{equation}
    		\label{eq:4.20}
    		\begin{aligned}
    			& 2 m[D(p+q)-D(p)-D(q)] b=0\\
    			& \implies [D(p+q)-D(p)-D(q)] b=0~(\text{By condition (iv) of Theorem \ref{der3}})\\
    			& \implies [D(p+q)-D(p)-D(q)]_{12} b=0\\
    			& \& ~ [D(p+q)-D(p)-D(q)]_{22} b=0\\
    			& \implies [D(p+q)-D(p)-D(q)]_{12} =0~(\text{By condition (b) of Definition \ref{def1}})\\
    			& \& ~ [D(p+q)-D(p)-D(q)]_{22} =0~(\text{By condition (ii) of Theorem \ref{der3}}).
    		\end{aligned}
    	\end{equation}

    	Let $s \in \mathcal{T}_{12}$. Then
    	\begin{equation}
    		\label{eq:4.21}
    		\begin{aligned}
    			(m+n) D((p+q) s)=2 m D(p+q) s+2 n(p+q) D(s).
    		\end{aligned}
    	\end{equation}
    	
    	Also,
    	\begin{equation}
    		\label{eq:4.22}
    		\begin{aligned}
    			(m+n) D((p+q) s)&=(m+n) D(0)\\
    			&=(m+n)(D(p s)+D(q s))\\
    			&=2 m D(p) s+2 n p D(s)+2 m D(q) s+2 n q D(s).
    		\end{aligned}
    	\end{equation}
    	
    	By \eqref{eq:4.21} and \eqref{eq:4.22},
    	\begin{equation}
    		\label{eq:4.23}
    		\begin{aligned}
    			& 2 m[D(p+q)-D(p)-D(q)] s=0\\
    			& \implies [D(p+q)-D(p)-D(q)] s=0~(\text{By condition (iv) of Theorem \ref{der3}})\\
    			& \implies [D(p+q)-D(p)-D(q)]_{11} s=0\\
    			& \implies [D(p+q)-D(p)-D(q)]_{11} =0~(\text{By condition (a) of Definition \ref{def1}}).
    		\end{aligned}
    	\end{equation}
    	
    	By \eqref{eq:4.20} and \eqref{eq:4.23}, we get the desired result.
    \end{proof}

    \begin{lemma}
    	\label{lem4.6}
    	Let $a_{1}, a_{2} \in \mathcal{T}_{11}$ and $b_{1}, b_{2} \in \mathcal{T}_{22}$. Then
    	\begin{equation}
    		\begin{aligned}
    			&(i)~D (a_{1}+a_{2} )=D (a_{1} )+D (a_{2} )\\
    			&(ii)~D (b_{1}+b_{2} )=D (b_{1} )+D (b_{2} ).
    		\end{aligned}
    	\end{equation}
    \end{lemma}

    \begin{proof}
    	Let $t_{2} \in \mathcal{T}_{22}$. Then
    	\begin{equation}
    		\label{eq:4.25}
    		\begin{aligned}
    			(m+n) D ( (a_{1}+a_{2} ) t_{2} )=2 m D (a_{1}+a_{2} ) t_{2}+2 n (a_{1}+a_{2} ) D (t_{2} ).
    		\end{aligned}
    	\end{equation}
    	
    	Also,
    	\begin{equation}
    		\label{eq:4.26}
    		\begin{aligned}
    			(m+n) D ( (a_{1}+a_{2} ) t_{2} )&=(m+n) D(0)\\
    			&=(m+n) (D (a_{1} t_{2} )+D (a_{2} t_{2} ))\\
    			&=2 m D (a_{1} ) t_{2}+2 n a_{1} D (t_{2} )+2 m D (a_{2} ) t_{2}+2 n a_{2} D (t_{2} ).
    		\end{aligned}
    	\end{equation}
    	
    	Comparing \eqref{eq:4.25} and \eqref{eq:4.26},
    	\begin{equation}
    		\label{eq:4.27}
    		\begin{aligned}
    			& 2m[D(a_1+a_2)-D(a_1)-D(a_2)]t_2=0 \\
    			& \implies [D(a_1+a_2)-D(a_1)-D(a_2)]t_2=0~(\text{By condition (iv) of Theorem \ref{der3}})\\
    			& \implies [D(a_1+a_2)-D(a_1)-D(a_2)]_{12}t_2=0\\
    			& \&~[D(a_1+a_2)-D(a_1)-D(a_2)]_{22}t_2=0\\
    			& \implies [D(a_1+a_2)-D(a_1)-D(a_2)]_{12}=0~(\text{By condition (b) of Definition \ref{def1}})\\
    			& \&~[D(a_1+a_2)-D(a_1)-D(a_2)]_{22}=0~(\text{By condition (ii) of Theorem \ref{der3}}).
    		\end{aligned}
    	\end{equation}
    	
    	Let $p \in \mathcal{T}_{12}$. Then
    	\begin{equation}
    		\label{eq:4.28}
    		\begin{aligned}
    			(m+n) D ((a_{1}+a_{2} ) p)=2 m D (a_{1}+a_{2} ) p+2 n (a_{1}+a_{2} ) D(p) .
    		\end{aligned}
    	\end{equation}
    	
    	Also,
    	\begin{equation}
    		\label{eq:4.29}
    		\begin{aligned}
    			(m+n) D ((a_{1}+a_{2} ) p)&=(m+n) D (a_{1} p+a_{2} p )\\
    			&=(m+n) [D (a_{1} p )+D (a_{2} p ) ]~(\text{By Lemma \ref{lem4.5}})\\
    			&=2 m D (a_{1} ) p+2 n a_{1} D(p)+2 m D (a_{2} )p+2 n a_{2} D(p).
    		\end{aligned}
    	\end{equation}
    	
    	Using \eqref{eq:4.28} and \eqref{eq:4.29},
    	\begin{equation}
    		\label{eq:4.30}
    		\begin{aligned}
    			& 2 m [D (a_{1}+a_{2} )-D (a_{1} )-D (a_{2} ) ] p=0\\
    			& \implies [D (a_{1}+a_{2} )-D (a_{1} )-D (a_{2} ) ] p=0 ~(\text{By condition (iv) of Theorem \ref{der3}})\\
    			& \implies [D (a_{1}+a_{2} )-D (a_{1} )-D (a_{2} ) ]_{11} p=0\\
    			& \implies [D (a_{1}+a_{2} )-D (a_{1} )-D (a_{2} ) ]_{11} =0~(\text{By condition (a) of Definition \ref{def1}})
    		\end{aligned}
    	\end{equation}
    	By \eqref{eq:4.27} and \eqref{eq:4.30}, we have $(i)$ of Lemma \ref{lem4.6}. Similarly, we can prove $(ii)$ of Lemma \ref{lem4.6}.
    \end{proof}

    \begin{lemma}
    	\label{lem4.7}
    	Let $a \in \mathcal{T}_{11}, p \in \mathcal{T}_{12}$ and $b \in \mathcal{T}_{22}$. Then
    	\begin{equation}
    		D (a+p+b )=D (a )+D(p)+D (b ).
    	\end{equation}
      \end{lemma}

    \begin{proof}
    	Let $a_{1} \in \mathcal{T}_{11}$. Then
    	\begin{equation}
    		\label{eq:4.32}
    		\begin{aligned}
    			(m+n) D ( (a+p+b ) a_{1} )=2 m D (a+p+b ) a_{1}+2 n (a+p+b ) D (a_{1} ).
    		\end{aligned}
    	\end{equation}
    	
    	Also,
    	\begin{equation}
    		\label{eq:4.33}
    		\begin{aligned}
    			(m+n) D ( (a+p+b ) a_{1} )&=(m+n) D (a a_{1} ) \\
    			&=(m+n) [D (a a_{1} )+D (p a_{1} )+D (b a_{1} ) ]\\
    			&=2 m D (a ) a_{1}+2 n a D (a_{1} )+2 m D(p) a_{1}\\&+2 n p D (a_{1} ) +2 m D (b ) a_{1}+2 n b D (a_{1} )
    		\end{aligned}
    	\end{equation}
    	
    	Comparing \eqref{eq:4.32} and \eqref{eq:4.33}, we have
    	\begin{equation}
    		\label{eq:4.34}
    		\begin{aligned}
    			& 2m [D (a+p+b )-D (a )-D(p)-D (b ) ] a_{1}=0 \\
    			& \implies [D (a+p+b )-D (a )-D(p)-D (b ) ] a_{1}=0 \\
    			&~(\text{By condition (iv) of Theorem \ref{der3}})\\
    			& \implies [D (a+p+b )-D (a )-D(p)-D (b ) ]_{11}a_{1}=0\\
    			& \implies [D (a+p+b )-D (a )-D(p)-D (b ) ]_{11}=0\\
    			&~(\text{By condition (i) of Theorem \ref{der3}}).
    		\end{aligned}
    	\end{equation}
    	
    	Let $t_{2} \in \mathcal{T}_{22}$. Then
    	\begin{equation}
    		\label{eq:4.35}
    		\begin{aligned}
    			(m+n) D ( (a+p+b ) t_{2} )=2 m D (a+p+b ) t_{2}+2 n (a+p+b ) D (t_{2} ).
    		\end{aligned}
    	\end{equation}
    	
    	Also,
    	\begin{equation}
    		\label{eq:4.36}
    		\begin{aligned}
    			(m+n) D ( (a+p+b ) t_{2} )&=(m+n) D (p t_{2}+b t_{2} )\\
    			&=(m+n) [D (p t_{2} )+D (b t_{2} ) ]~(\text{By Lemma \ref{lem4.2}})\\
    			&=(m+n) [D (a t_{2} )+D (p t_{2} )+D (b t_{2} ) ]\\
    			&=2 m D (a ) t_{2}+2 n a D (t_{2} )+2 m D(p) t_{2}\\
    			&+2 n p D (t_{2} ) +2 m D (b ) t_{2}+2 n b D (t_{2} ).
    		\end{aligned}
    	\end{equation}
    	
    	Comparing \eqref{eq:4.35} and \eqref{eq:4.36}, we have
    	\begin{equation}
    		\label{eq:4.37}
    		\begin{aligned}
    			& 2m [D (a+p+b )-D (a )-D(p)-D (b ) ] t_{2}=0\\
    			& \implies [D (a+p+b )-D (a )-D(p)-D (b ) ] t_{2}=0\\
    			& ~(\text{By condition (iv) of Theorem \ref{der3}})\\
    			& \implies [D (a+p+b )-D (a )-D(p)-D (b ) ]_{12} t_{2}=0\\
    			& \&~ [D (a+p+b )-D (a )-D(p)-D (b ) ]_{22} t_{2}=0\\
    			& \implies [D (a+p+b )-D (a )-D(p)-D (b ) ]_{12} =0\\
    			& ~(\text{By condition (b) of Definition \ref{def1}})\\
    			& \&~ [D (a+p+b )-D (a )-D(p)-D (b ) ]_{22} =0\\
    			& ~(\text{By condition (ii) of Theorem \ref{der3}}).
    		\end{aligned}
    	\end{equation}
    	By \eqref{eq:4.34} and \eqref{eq:4.37}, we get the desired result.
    \end{proof}

    Now, we are ready to prove Theorem \ref{der3}.

    \begin{proof}[Proof of Theorem \ref{der3}]
    	Let $a, b \in \mathcal{T}$. Then $a=a_{11}+a_{12}+a_{22}$ and $b=b_{11}+b_{12}+b_{22}$ where $a_{ij}$, $b_{ij}\in \mathcal{T}_{ij}$.
    	
    	\begin{equation}
    		\begin{aligned}
    			D(a+b) &= D (a_{11}+a_{12}+a_{22}+b_{11}+b_{12}+b_{22} )\\
    			&=D (a_{11}+b_{11}+a_{12}+b_{12}+a_{22}+b_{22} )\\
    			&=D (a_{11}+b_{11} )+D (a_{12}+b_{12} )+D (a_{22}+b_{22} )~(\text{By Lemma \ref{lem4.7}})\\
    			&=D (a_{11} )+D (b_{11} )+D (a_{12} )+D (b_{12} )+D (a_{22} )+D (b_{22} )\\
    			& ~(\text{By Lemma \ref{lem4.5} and \ref{lem4.6}} )
    			\\
    			&=D (a_{11} )+D (a_{12} )+D (a_{22} )+D (b_{11} )+D (b_{12} )+D (b_{22} )\\
    			&=D (a_{11}+a_{12}+a_{22} )+D (b_{11}+b_{12}+b_{22} ) ~(\text{By Lemma \ref{lem4.7}} )\\
    			&=D(a)+D(b).
    		\end{aligned}
    	\end{equation}
    	
    	Hence, $D$ is additive. Thus, $D$ is a $(m,n)$-derivation.
    \end{proof}

 \section*{Acknowledgement}

The first author is thankful to the University Grants Commission (UGC), Govt. of India for financial supports under Sr. No-2121540952, Ref. No. 20/12/2015(ii)EU-V dated 31/08/2016 and all authors are thanful to the Indian Institute of Technology Patna for providing the research facilities.

\end{document}